\documentclass[11pt, a4paper, fleqn]{article}
\usepackage[latin1,utf8]{inputenc}
\usepackage{amsmath}
\usepackage{amsthm}
\usepackage{amsfonts}
\usepackage{amssymb}
\usepackage[T1]{fontenc}
\usepackage[all]{xy}
\usepackage{bbm}
\usepackage{enumitem}
\usepackage{color}
\usepackage[hidelinks]{hyperref}
\usepackage{graphicx}
\usepackage{transparent}
\usepackage{graphics}
\usepackage{xcolor}
\usepackage{color}
\usepackage{geometry}
\usepackage{marginnote}

\usepackage{pgf,tikz}
\usepackage{mathrsfs}
\usetikzlibrary{arrows,decorations.pathmorphing,backgrounds,positioning,fit,petri,cd}

\title{Stability Conditions and Maximal Green Sequences in Abelian Categories}

\author{Thomas Br\"ustle, David Smith and Hipolito Treffinger}

\theoremstyle{plain} 
\newtheorem{theorem}{Theorem}[section]
\newtheorem{prop}[theorem]{Proposition}
\newtheorem{lem}[theorem]{Lemma}
\newtheorem{cor}[theorem]{Corollary}

\theoremstyle{remark}
\newtheorem{rmk}[theorem]{Remark}
\newtheorem{ex}[theorem]{Example}

\theoremstyle{definition}
\newtheorem{defi}[theorem]{Definition}
\newtheorem{conj}[theorem]{Conjecture}

\def\Hom{\mbox{Hom}}
\def\Obj{\mbox{\rm Obj}}
\def\Filt{\mbox{Filt}}

\def\coker{\mbox{coker}}
\def\im{\mbox{im}}
\def\Fac{\mbox{Fac}\,}

\newcommand{\A}{\mathcal{A}}
\newcommand{\X}{\mathcal{X}}
\newcommand{\F}{\mathcal{F}}
\renewcommand{\P}{\mathcal{P}}
\newcommand{\T}{\mathcal{T}}
\newcommand{\ra}{\rightarrow}

\newcommand{\D}{\mathfrak{D}}

\newcommand{\mP}{\mathbb{P}}
\newcommand{\N}{\mathbb{N}}
\newcommand{\R}{\mathbb{R}}

\usepackage[nottoc]{tocbibind}

\begin{document}

\maketitle

\abstract{In this paper we study the stability functions on abelian categories introduced by Rudakov in \cite{Ru} and their relation with torsion classes and maximal green sequences. 
Moreover we introduce a new kind of stability function which is compatible with the wall and chamber of the category.}

%%%%%%%%%%%%%%%%%%%%%%%%%%%%%%%%%%%%%%
%%%%%%%%%%%%%%%%%%%%%%%%%%%%%%%%%%%%%%
\section{Introduction}
%%%%%%%%%%%%%%%%%%%%%%%%%%%%%%%%%%%%%%
%%%%%%%%%%%%%%%%%%%%%%%%%%%%%%%%%%%%%%
The concept of stability condition was introduced in algebraic geometry by Mumford in \cite{Mumford1965} to study moduli spaces under the actions of a group.
The success of this new approach was such that many mathematicians started adapting these tools to different branches of mathematics.
In the case of representation theory of quivers, they were introduced in seminal papers by Schofield \cite{Scho} and King \cite{Ki}, and the general notion of stability was formalised in the context of abelian categories by Rudakov \cite{Ru}.

We  study Rudakov's notion of stability on an abelian length category $\mathcal{A}$, which is given by a function  $\phi: \Obj^*(\A) \to \P$ over $\A$ that assigns to each non-zero object $X$ a \textit{phase} $\phi(X)$, which is an element of a totally ordered set $\P$, satisfying the so-called see-saw condition on short exact sequences, see
definition \ref{seesaw}.
A non-zero object $M$ in $\mathcal{A}$ is said to be \textit{$\phi$-stable} (or \textit{$\phi$-semistable}) if every non-trivial subobject $L\subset M$ satisfies $\phi(L)< \phi(M)$ ( or $\phi(L)\leq \phi(M)$, respectively).
Inspired by \cite{B16}, but in the more general context of abelian categories allowing infinitely many simple objects, we then define for each phase $p$ a torsion pair $(\T_p, \F_p)$ in $\mathcal{A}$ as follows (see Proposition \ref{torsionpair}):

$$\T_p=\{M\in\A \ : \phi(N)\geq p \text{ for every quotient $N$ of $M$}\}$$
$$\F_p=\{M\in\A \ : \ \phi(N)< p \text{ for every subobject $N$ of $M$}\}$$

Since $\T_p\supseteq\T_q$ when $p\leq q$,  a stability function $\phi$ induces a chain of torsion classes in $\A$.  
We adapt the definition of maximal green sequence introduced by Keller in \cite{KellerMaximalGreenSequences} for cluster algebras to the context of abelian categories. 
In this context, a maximal green sequence in an abelian category $\A$ is a finite not refinable increasing chain of torsion classes starting with the zero class and ending in $\A$.
The equivalence of this definition to the original on cluster algebras is shown in \cite[Proposition 4.9]{BSTw&c} using $\tau$-tilting theory
.
Following Engenhorst \cite{E14}, we call a stability function $\phi$ on $ \A$ \textit{discrete} if it admits (up to isomorphism) at most one $\phi$-stable object for every phase at $p$. 

The first main result of this paper characterises which stability functions induce maximal green sequences in $\A$.
\begin{theorem}[Theorem \ref{maximalgreensequences}]
Let $\phi:\A\ra \P$ be a stability function that admits no maximal phase. Then $\phi$ induces a maximal green sequence of torsion classes in $\A$ if and only if $\phi$ is a discrete stability function inducing only finitely many different torsion classes $\T_p$.
\end{theorem}

\bigskip
The wall and chamber structure of a module category has been introduced by Bridgeland in \cite{B16} to give an algebraic interpretation of scattering diagrams studied in mirror symmetry by Gross, Hacking, Keel and Kontsevich, see \cite{GHKK}. 
It has been shown in \cite{BSTw&c} that all functorially finite torsion classes of an algebra can be obtained from its wall and chamber structure. 
We consider in this paper more generally abelian categories $\A$ with finitely many simple objects. In this context, we provide a construction of stability functions on $\A$  that conjecturally induce all its maximal green sequences.
These stability functions are induced by certain curves, called \textit{red paths} in the wall and chamber structure of $\A$.
In particular we show that red paths give a non-trivial compatibility between the stability conditions introduced by King in \cite{Ki} and the stability functions introduced by Rudakov in \cite{Ru}.
As a consequence, we show that the wall and chamber structure of an algebra can be recovered using red paths, see Theorem \ref{paths}.
\bigskip

This paper is a revised version of one part of a preprint \cite{BST}. 
We would like to point to the paper \cite{BCZ} by Barnard, Carrol and Zhu  which obtains proposition \ref{mutation} in the context of module categories of finite dimensional algebras over an algebraically closed field. 

We refer to the textbooks \cite{ARS,AsSS,bookRalf} for background material.

\paragraph{Acknowledgements.}
The authors want to thank Kiyoshi Igusa for his pertinent remarks and helpful discussions.
They also thank Patrick Le Meur, whose comments on the PhD thesis of the third author lead to the current version of this work.

The first and the second author were supported by Bishop's University, Université de Sherbrooke and NSERC of Canada, and the third author was supported by the EPSRC funded project EP/P016294/1.

%%%%%%%%%%%%%%%%%%%%%%%%%%%%%%%%%%%%%%
%%%%%%%%%%%%%%%%%%%%%%%%%%%%%%%%%%%%%%
\section{Stability Conditions}
%%%%%%%%%%%%%%%%%%%%%%%%%%%%%%%%%%%%%%
%%%%%%%%%%%%%%%%%%%%%%%%%%%%%%%%%%%%%%
The aim of this section is to study Rudakov's \cite{Ru} definition of stability on abelian categories. While \cite{Ru} uses the notion of a proset, we prefer to work with stability functions. 
We first review this concept of stability here, and then discuss torsion classes arising from a stability function.

%%%%%%%%%%%%%%%%%%%%%%%%%%%%%%%%%%%%%%
\subsection{Stability functions}
Throughout this section, we consider an essentially small abelian category $\mathcal{A}$.
We denote by $\Obj^*(\A)$ the class of non-zero objects of $\A$.

\begin{defi}\label{seesaw}
Let %$\A$ be an abelian category and
$(\mathcal{P},\leq)$ be a totally ordered set.
A function $\phi :  \Obj^*(\mathcal{A}) \to \mathcal{P}$ 
%from $\Obj^*(\A)$ the non-zero objects of $\A$ to $\P$ 
which is constant on isomorphism classes is said to be a \textit{stability function} if for each short exact sequence  $0 \to L \to M \to N \to 0$ of non-zero objects in $\mathcal{A}$ one has the so-called \textit{see-saw} (or \textit{teeter-totter}) property, that is exactly one of the following holds:
$$
\begin{array}{ll}
\text{either} & \phi(L) < \phi(M) < \phi(N), \\
\text{or} & \phi(L) > \phi(M) > \phi(N),\\
\text{or} & \phi(L) = \phi(M) = \phi(N).
\end{array}
$$
\end{defi}

For a non-zero object $x$ of $\A$, we refer to $\phi(x)$ as the \textit{phase} (or \textit{slope}) of $x$.

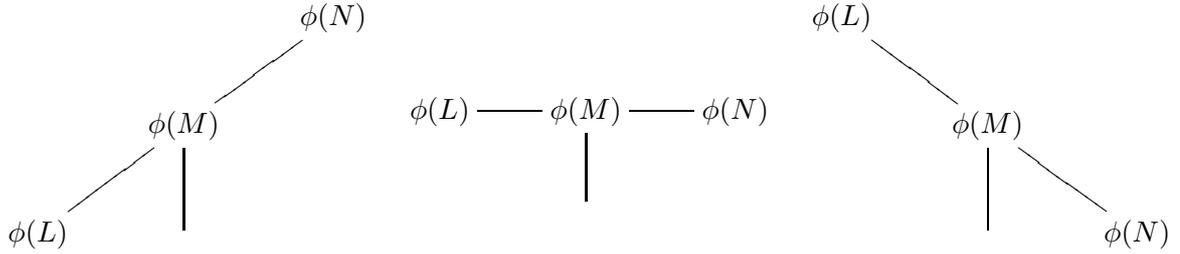
\begin{figure}
$$
\begin{array}{ccc}
\xymatrix{ & &\phi(N) \\
           &\phi(M)\ar@{-}[d]\ar@{-}[ru] & \\
           \phi(L)\ar@{-}[ru]& &}
&\xymatrix{ & & \\
           \phi(L)\ar@{-}[r]&\phi(M)\ar@{-}[d]\ar@{-}[r] &\phi(N) \\
           & &}
&\xymatrix{\phi(L)\ar@{-}[dr] & & \\
           &\phi(M)\ar@{-}[d]\ar@{-}[dr] & \\
           & &\phi(N)}
\end{array}
$$
\caption{The see-saw (or teeter-totter) property.}
\label{fig:seesaw}
\end{figure}

\begin{rmk}
Note that the image by $\phi$ of the zero object in $\A$ is not well defined if there exist two non-zero objects $M$ and $N$ such that $\phi(M)\neq \phi(N)$.
Indeed, is enough to take the following short exact sequences.
$$
\begin{array}{lcr}
0 \to 0 \to M \to M \to 0 & & 0 \to 0 \to N \to N \to 0
\end{array}
$$
Then, applying the see-saw property twice we conclude that $\phi(0)=\phi(M)$ and $\phi(N) = \phi(0)$.
Since by hypothesis $\phi(M) \neq \phi(N)$ we have that $\phi(0)$ is not well defined.
That is the reason why stability functions should be defined only over the non-zero objects of the category.
\bigskip

Note that Rudakov defined  stability structures using the notion of a proset, that is, a pre-order $\prec$ on $\Obj^*(\mathcal{A})$ satisfying for all $L,M$ in $\Obj^*(\mathcal{A}) $  that $L \prec M$ or $M \prec L$, or both. 
We can define an equivalence relation on $\Obj^*(\mathcal{A}) $ by setting $L \sim M$ when both $L \prec M$ and $M \prec L$ are satisfied, and denote by $\mathcal{P} = \Obj^*(\mathcal{A}) / \sim$ the set of equivalence classes. 
The pre-order $\prec$ thus turns $\mathcal{P}$ into a totally ordered set, whose order relation we denote by $\le$.
The projection $\phi: \Obj^*(\mathcal{A}) \to \mathcal{P} $ that assigns to each object its equivalence class is then the function from Definition \ref{seesaw}, and the notion of stability we consider here is equivalent to Rudakov's original formulation.
\end{rmk}
\bigskip

The stability functions as defined above generalise several notions of stability conditions present in the literature as we can see in the following remarks. 

\begin{rmk}
In \cite{Ki}, King adapted the Geometric Invariant Theory introduced by Mumford in \cite{Mumford1965} to the context of abelian categories with Grothendieck group of finite rank. 
In \cite[Proposition 3.4]{Ru}, Rudakov shows that every stability condition as defined by King induces a stability function. 
\end{rmk}

\begin{rmk}
Stability functions are present in the physics literature, and in this case they are induced by a central charge $Z$. We recall this notion here, following the treatment given in \cite{B}:

A \textit{linear stability function} on an abelian category $\mathcal{A}$ is given by a central charge, that is, a group homomorphism $Z:K(\mathcal{A})\rightarrow \mathbb{C}$ on the Grothendieck group $K(\mathcal{A})$ such that for all $0\neq M\in \mathcal{A}$ the complex number $Z(M)$ lies in the strict upper half-plane 
$$\mathbb{H}=\{r \cdot\text{exp}(i\pi\phi):r>0\text{ and } 0 <\phi \leq 1\}.$$
Given such a central charge $Z:K(\mathcal{A})\rightarrow \mathbb{C}$, the phase of an object $0\neq M\in \mathcal{A}$ is defined to be 
$$\phi(M)=(1/\pi)\text{arg}Z(M).$$
A simple argument on the sum of vectors in the plane shows that the phase function $\phi: \Obj^*(\mathcal{A})\rightarrow (0,1]$ satisfies the see-saw property.
\end{rmk}
\bigskip

The most important feature of a stability function $\phi$ is the fact that they create a distinguished subclass of objects in $\A$ called \textit{$\phi$-semistables}.
They are defined as follows.

\begin{defi}\cite[Definition 1.5 and 1.6]{Ru} \label{defss}
Let $\phi:\Obj^*(\mathcal{A})\rightarrow \mathcal{P}$ be a stability function on $\mathcal{A}$. 
A non-zero object $M$ of $\A$ is said to be \textit{$\phi$-stable} (or \textit{$\phi$-semistable}) if every non-trivial subobject $L\subset M$ satisfies $\phi(L)< \phi(M)$ ( or $\phi(L)\leq \phi(M)$, respectively).
\end{defi}

\begin{rmk}
Note that, due to the see-saw property, one can equally define the $\phi$-stable (or $\phi$-semistable) objects as those objects $M$ whose quotient objects $N$ satisfy  $\phi(N) > \phi(M)$ (or $\phi(N)\geq \phi(M)$, respectively).
\end{rmk}

The following theorem from \cite{Ru} implies that morphisms between $\phi$-semistable objects respect the order induced by $\phi$, that is, $\Hom_{\A}(M,N)=0$ whenever  $M,N$ are $\phi$-semistable and  $\phi(M)>\phi(N)$.  

\begin{theorem}\cite[Theorem 1]{Ru}\label{homzero}
Let $\phi:\Obj^*(\mathcal{A})\rightarrow \mathcal{P}$ be a stability function on $\mathcal{A}$ and let $f: M \to N$ be a non-zero morphism in $\mathcal{A}$ between two $\phi$-semistable objects $M,N$ such that $\phi(M) \geq \phi(N)$. Then
\begin{itemize}
\item[(a)] $\phi(M) = \phi(N)$.
\item[(b)] If $N$ is $\phi$-stable then $f$ is an epimorphism.
\item[(c)] If $M$ is $\phi$-stable then $f$ is a monomorphism.
\item[(d)] If $M$ and $N$ are both $\phi$-stable then $f$ is an isomorphism.
\end{itemize}
\end{theorem}

\begin{cor}\label{noniso}
Let $\phi:\Obj^*(\mathcal{A})\rightarrow \mathcal{P}$ be a stability function on $\mathcal{A}$ and let $M,N\in\A$ be two non-isomorphic $\phi$-stable objects such that $\phi(M)=\phi(N)$.
Then $$\emph{Hom}_{\A}(M,N)=0.$$
\end{cor}

\begin{rmk}\label{bricks}
As observed in \cite{Ru}, Theorem \ref{homzero} implies that $\phi$-stable objects are bricks when $\A$ is a Hom-finite $k$-category over an algebraically closed field $k$. 
Here $M$ is called a \textit{brick} when End$(M) \simeq k$. 
This implies in particular that $\phi$-stable objects are indecomposable. 
In fact, it is easy to see that $\phi$-stable objects are always indecomposable, for any abelian category $\A$.
\end{rmk}

%%%%%%%%%%%%%%%%%%%%%%%%%%%%%%%%%%%%%%
\subsection{Harder-Narasimhan filtration and stability functions} 
From now on, we assume that the abelian category $\mathcal A$ is a \textit{length category}, that is, each object $M$ admits a filtration 
$$0=M_0\subsetneq M_1\subsetneq M_2\subsetneq \dots \subsetneq M_{l-1} \subsetneq M_l=M$$
such that the quotients $M_i/M_{i-1}$ are simple. In particular, $\mathcal{A}$ is both noetherian and artinian. For a finite dimensional $k$-algebra $A$ over a field  $k$, the category mod\;$A$ of finitely generated $A$-modules is a length category.

We borrow the following terminology from \cite{B}, however the concept was already used in \cite{Ru}.

\begin{defi}\label{defMDQ}
Let $\mathcal{A}$ be an abelian length category and let $M$ be a non-zero object of $\A$.
\begin{enumerate}[label=(\alph*)]
\item A pair $(N,p)$ consisting of a non-zero object $N\in\mathcal{A}$ and an epimorphism $p:M\rightarrow N$ is said to be a \textit{maximally destabilising quotient} (or m.d.q. for short) of $M$ if for every epimorphism $p':M\rightarrow N'$ satisfies $\phi(N')\geq\phi(N)$, and moreover, if  $\phi(N)=\phi(N')$, then the morphism $p'$ factors through $p$. 
\item A pair $(L,i)$ consisting of a non-zero object $L\in\mathcal{A}$ and a monomorphism $i:L\rightarrow M$ is a \textit{maximally destabilising subobject } (or m.d.s. for short) of $M$ if for every monomorphism $i':L'\rightarrow M$ satisfies $\phi(L')\leq\phi(L)$, and moreover, if  $\phi(L)=\phi(L')$ then the morphism $i'$ factors through $i$.
\end{enumerate}
\end{defi}

We sometimes omit the epimorphism $p$ when referring to a maximally destabilising quotient, similarly for maximally destabilising subobjects.

\begin{rmk}
Note that if follows directly from Definition \ref{defMDQ} that $\phi(L) \leq \phi(M) \leq \phi(N)$, where $L$ and $N$ are the maximally destabilising subobject and quotient of $M$, respectively.
This property will be used often in the proofs of the present paper.
\end{rmk}

An important property of maximally destabilising subobjects and quotients is that they are always $\phi$-semistable.

\begin{lem}\label{MDQunicity}
Let  $\phi:\Obj^*(\mathcal{A})\rightarrow \mathcal{P}$ be a stability function on $\mathcal{A}$ and let $M$ be a non-zero object in $\mathcal{A}$. Then:
\begin{enumerate}[label=(\alph*)]
\item[(a)] The maximally destabilising object $(N,p)$ of $M$ is $\phi$-semistable and unique up to isomorphism;
\item[(b)] The maximally destabilising subobject $(L,i)$ of $M$ is $\phi$-semistable and unique up to isomorphism.
\end{enumerate}
\end{lem}

\begin{proof}
This is a follows directly from \cite[Proposition 1.9]{Ru} and its dual. 
\end{proof}
 
The following theorem from \cite{Ru} implies in particular that  every object admits a maximally destabilising quotient and a maximally destabilising subobject.

\begin{theorem}\cite[Theorem 2, Proposition 1.13]{Ru}\label{HN}
Let $\mathcal{A}$ be an abelian length category with a stability function $\phi:\Obj^*\mathcal{A}\rightarrow \mathcal{P}$, and let $M$ be a non-zero object in $\mathcal{A}$.  
Up to isomorphism, $M$ admits a unique \textit{Harder-Narasimhan filtration}, that is a filtration 
$$0=M_0\subsetneq M_1\subsetneq M_2\subsetneq \dots \subsetneq M_{n-1} \subsetneq M_n=M$$
such that 
\begin{enumerate}[label=(\alph*)]
\item the quotients $F_i=M_i/M_{i-1}$ are $\phi$-semistable, 
\item $\phi(F_n)<\phi(F_{n-1})<\dots<\phi(F_2)<\phi(F_1)$.
\end{enumerate}
Moreover, $F_1=M_1$ is the maximally destabilising subobject of $M$ and $F_n=M_n/M_{n-1}$ is the maximally destabilising quotient of $M$.
\end{theorem}

For further use, it is also worthwhile to recall the following weaker version of a result from Rudakov.

\begin{theorem}\cite[Theorem 3]{Ru}\label{JH}
Let $\mathcal{A}$ be an abelian length category with a stability function $\phi:\Obj^*\mathcal{A}\rightarrow \mathcal{P}$, and let $M$ be a $\phi$-semistable object in $\mathcal{A}$.  
There exists a filtration 
$$0=M_0\subsetneq M_1\subsetneq M_2\subsetneq \dots \subsetneq M_{n-1} \subsetneq M_n=M$$
such that 
\begin{enumerate}[label=(\alph*)]
\item the quotients $G_i=M_i/M_{i-1}$ are $\phi$-stable, 
\item $\phi(M)=\phi(G_{n-1})=\dots=\phi(G_2)=\phi(G_1)$.
\end{enumerate}
Moreover, the Jordan-H\"older property holds, in the sense that the set $\{G_i\}$ of factors is uniquely determined up to isomorphism.
\end{theorem}

%%%%%%%%%%%%%%%%%%%%%%%%%%%%%%%%%%%%%%
\subsection{Torsion pairs}\label{Sect:torsion} 

The concept of torsion pair in an abelian category was first introduced by Dickson in \cite{Dickson1966}, generalising properties of abelian groups of finite rank. 
The definition is the following.

\begin{defi}
Let $\A$ be an abelian category. 
Then the pair $(\T, \F)$ of full subcategories of $\A$ is a \textit{torsion pair} if the following conditions are satisfied:
\begin{itemize}
    \item $\Hom_\A (X,Y)=0$ for all $X \in \T$ and $Y \in \F$;
    \item If $X \in \A$ is such that $\Hom_\A(X,Y)=0$ for all $Y \in \F$ then $X\in \T$;
    \item If $Y \in \A$ is such that $\Hom_\A(X,Y)=0$ for all $X\in \T$ then $Y \in \F$.
\end{itemize}
Given a torsion pair $(\T, \F)$ we say that $\T$ is a torsion class and $\F$ is a torsion free class. 
\end{defi}

It is well-known that a subcategory $\mathcal{T}$ of $\mathcal{A}$ is the torsion class of a torsion pair $(\mathcal{T},\mathcal{F})$ if and only if $\mathcal{T}$ is closed under quotients and extensions. 
Dually, a subcategory $\mathcal{F}$ of $\mathcal{A}$ is the torsion-free class of a torsion pair if and only if $\mathcal{F}$ is closed under subobjects and extensions.
See \cite[Proposition VI.1.4]{AsSS} for more details. 

In this subsection, we show that a stability function $\phi: \Obj^* \A \to \mathcal P$  induces a torsion pair $(\mathcal{T}_p, \mathcal{F}_p)$ in $\mathcal{A}$ for every $p\in\mathcal{P}$ where
$$\mathcal{T}_{p}=\{M\in \Obj^*(\mathcal A): \phi(M')\geq p \text{ where $M'$ is the m.d.q. of }M\} \cup \{0\}$$
$$\mathcal{F}_{p}=\{M\in \Obj^*(\mathcal A): \phi(M'') < p \text{ where $M''$ is the m.d.s. of }M\} \cup \{0\}$$
But before doing so, we need to fix some notation. 

\begin{defi}
Let $\phi : \A \to \P$ be a stability function and let $p \in \P$. 
We define $\A_{\geq p}$ to be
$$\mathcal{A}_{\geq p}:=\{0\} \cup \{M\in {\mathcal A}: M\text{ is $\phi$-semistable and }\phi(M)\geq p\}.$$
we define in a similar way $\A_{\leq p}$, $\A_{> p}$, $\A_{< p}$ and $\A_p$.
\end{defi}

Given a subcategory $\X$ and an object $M$ of $A$ we say that $M$ is \textit{filtered} by $\X$ if there exists a chain of nested subobjects 
$$0=M_0\subsetneq M_1\subsetneq M_2\subsetneq \dots \subsetneq M_{n-1} \subsetneq M_n=M$$
of $M$ such that $M_i / M_{i-1}$ is an object of $\X$.
Moreover we denote $\Filt(\X)$ to be the full subcategory of $\A$ defined as follows.
$$\Filt(\X):= \{M \in \A : M \text{ is filtered by $X$}\}$$

The following proposition not only shows that $\T_p$ is a torsion class, but also gives a series of equivalent characterisations. 

\begin{prop}\label{eqdeftorsion}
Let $\phi: \A\ra \P$ be a stability function and consider the full subcategory $\T_p$ of $\A$ to be 
$$\mathcal{T}_{p}=\{M\in \Obj^*(\mathcal A): \phi(M')\geq p \text{ where $M'$ is the m.d.q. of }M\} \cup \{0\}.$$
Then:
\begin{enumerate}
\item $\T_p$ is a torsion class;
\item $\mathcal{T}_{p}=\text{\emph{Filt}}(\mathcal{A}_{\geq p})$;
\item $\T_p=\text{\emph{Filt}}(\text{\emph{Fac}}\mathcal{A}_{\geq p})$;
\item $\T_p=\{M\in\A \ : \phi(N)\geq p \text{ for every quotient $N$ of $M$}\}$.
\end{enumerate}
\end{prop}

\begin{proof}
\textit{1.} We need to show that $\mathcal{T}_{p}$ is closed under extensions and quotients. 

To show that $\mathcal{T}_{p}$ is closed under extensions, suppose that  
$$0\rightarrow L\overset{f}\rightarrow M\overset{g}\rightarrow N\rightarrow 0$$ is a short exact sequence in $\mathcal{A}$ with $L, N\in\mathcal{T}_{p}$.  
Let $(M',p_M)$ be the maximally destabilising quotient of $M$. 
Then we can construct the following commutative diagram.
$$\xymatrix{
 0\ar[r]& L\ar[r]^f\ar[d] &M\ar[r]^g\ar[d]^{p_M} &N\ar[r]\ar[d] &0 \\
 0\ar[r]& \im (p_Mf)\ar[r]^{f'}\ar[d] & M'\ar[r]^{g'}\ar[d] &\coker f'\ar[r]\ar[d] &0 \\
  & 0 & 0 & 0 & }$$
Let $(L',p_L)$ and $(N',p_N)$ be the maximally destabilising quotients of $L$ and $N$ respectively.

If $\im (p_Mf)=0$, then there exists an epimorphism $h:N\rightarrow M'$, and it follows from the definition of $N'$ that $\phi(M')\geq\phi(N')\geq p$.
Else, it follows from the semistability of $M'$ that $\phi(\im (p_Mf))\leq \phi(M')$. 
Moreover, $\phi(\im(p_Mf))\geq\phi(L')\geq p$ since $L'$ is a maximally destabilising quotient. 
Consequently $\phi(M')\geq p$ and $\mathcal{T}_P$ is closed under extensions.

To show that $\mathcal{T}_p$ is closed under quotients, suppose that $f:M\rightarrow N$ is an epimorphism with $M\in\mathcal{T}_p$. 
Let $(M',p_M)$ and $(N',p_N)$ be the maximally destabilising quotients of $M$ and $N$ respectively. 
Then $p_Nf:M\rightarrow N'$ is an epimorphism and it follows from the definition of $M'$ that $\phi(N')\geq\phi(M')\geq p$.
Hence $N \in \T_p$. 
This proves that $\T_p$ is a torsion class.

\textit{2. and 3.} 
Clearly, $\text{Filt}(\mathcal{A}_{\geq p})\subseteq \text{Filt}(\Fac\mathcal{A}_{\geq p})$. 
On the other hand, it follows from \cite[Proposition 3.3]{DIJ} that $\text{Filt}(\Fac\mathcal{A}_{\geq p})$ is the smallest torsion class containing $\mathcal{A}_{\geq p}$. 
As $\mathcal{A}_{\geq p}\subseteq\mathcal{T}_p$, we get $\text{Filt}(\Fac\mathcal{A}_{\geq p})\subseteq\mathcal{T}_p$. 

It thus remains to show that $\mathcal{T}_p\subseteq \text{Filt}(\mathcal{A}_{\geq p})$. 
Let $M\in\mathcal{T}_{p}$, and let $M'$ be a maximally destabilising quotient of $M$.  
By definition of $\mathcal{T}_p$, we have that $\phi(M')\geq p$. 
Therefore we can consider the Harder-Narasimhan filtration of $M$ and Theorem \ref{HN} implies that $M\in\text{Filt}(\mathcal{A}_{\geq p})$. 
Hence $\T_p\subseteq \text{Filt}(\mathcal{A}_{\geq p})\subseteq\text{Filt}(\Fac(\A_{\geq p}))\subseteq \T_p$.

\textit{4.} 
Let $M\in\mathcal{T}_p$, and suppose that $M'$ is its maximally destabilising quotient.  
By definition of the maximally destabilising quotient, every quotient $N$ of $M$ is such that $\phi(N)\geq\phi(M')\geq p$.  
Thus $\mathcal{T}_p\subseteq \{M\in\A \ : \ \phi(N)\geq p \text{ for every quotient $N$ of $M$}\}$.  
The reverse inclusion is immediate.
\end{proof}

The following result is the dual statement for the torsion-free class $\mathcal{F}_p$.

\begin{prop}\label{eqdeftorsionfree}
Let $\phi: \A\ra \P$ be a stability function and consider the full subcategory $\F_p$ of $\A$ to be $$\mathcal{F}_{p}=\{M\in \Obj^*(\mathcal A): \phi(M'') < p \text{ where $M''$ is the m.d.s. of }M\} \cup \{0\}.$$
 Then:
\begin{enumerate}[label=(\alph*)]
\item $\F_p$ is a torsion free class;
\item $\mathcal{F}_{p}=\text{\emph{Filt}}(\mathcal{A}_{< p})$; %\textcolor{blue}{(We need to define $\mathcal{A}_{< p}$ somewhere...)}
\item $\F_p=\text{\emph{Filt}}(\text{\emph{Sub}}\mathcal{A}_{< p})$.
\item $\F_p=\{M\in\A \ : \ \phi(N)< p \text{ for every subobject $N$ of $M$}\}$.
\end{enumerate}
\end{prop}

Now were are able to  prove the main result of this section.

\begin{prop}\label{torsionpair}
Let $p\in\mathcal{P}$.  
Then $(\mathcal{T}_p, \mathcal{F}_p)$ is a torsion pair in $\mathcal{A}$.
\end{prop}
\begin{proof}
We first show that $\Hom_{\mathcal{A}}(\mathcal{T}_p, \mathcal{F}_p)=0$.
Suppose that $f\in\Hom_{\A}(M,N)$, where $M\in\T_p$ and $N\in\F_p$.  
Let $M'$ be the maximally destabilising quotient of $M$ and $N'$ be the maximally destabilising subobject of $N$.  
Then $\im f$ is a quotient of $M$ and a subobject of $N$. 
So, if $f\neq 0$, it follows from the definitions of $M'$ and $N'$ that $\phi(\im f)\geq \phi(M')\geq p$ and $\phi(\im f)\leq\phi(N')<p$, a contradiction.  
Thus $f=0$ and $\Hom_{\mathcal{A}}(\mathcal{T}_p, \mathcal{F}_p)=0$.

For the maximality, suppose for instance that $\Hom_\A(\T_p, N)=0$. 
If $N'$ is the maximally destabilising subobject of $N$, it follows that $\Hom_\A(\T_p, N')=0$, and thus $\phi(N')<p$ by definition of $\T_p$. 
Consequently, $N\in\F_p$.  
We show in the same way that $\Hom_\A(M, \F_p)=0$ implies $M\in\T_p$, which proves maximality.
\end{proof}

As a consequence of the previous proposition we have the following result that provides a method to build abelian subcategories of $\A$ using stability conditions. 

\begin{prop}\label{fix-slope}
Let $\phi: \A \to \P$ be a stability function and $p\in \mathcal{P}$ be fixed. Then the full subcategory 
$$\mathcal{A}_{p}=\{0\} \cup \{M\in {\mathcal A}: M\text{ is $\phi$-semistable and }\phi(M)=p\}$$
is a wide subcategory of $\A$.
\end{prop}

\begin{proof}
$\A_p$ is a wide subcategory if it is abelian. 
To show that, we note first that $\A_p = \T_p \cap \text{Filt}(\A_{\leq p})$.
Then Proposition \ref{eqdeftorsion} and its dual impliy that $\A_p$ is the intersection of a torsion class $\T_p$ and a torsion free class $\text{Filt}(\A_{\leq p})$.
This implies in particular that $\A_p$ is closed under extensions.

Now we show that $\mathcal{A}_{p}$ is closed under taking kernels and cokernels.
Let $f:M\to N$ be a morphism in $\mathcal{A}_{p}$.  
If $f$ is zero or an isomorphism, the result follows at once.  
Otherwise, consider the following short exact sequences in $\mathcal A$
$$ 0 \to \ker f \to M \to \im f \to 0$$
$$ 0 \to \im f \to N \to \coker f \to 0$$
where all these objects are non-zero. 
The semistability of $M$ implies $\phi(\im f) \geq \phi(M)=p$, while the semistability of $N$ implies $\phi(\im f)\leq \phi(N)=p$. 
Consequently $\phi(\im f)=p$. 
The see-saw property applied to the two exact sequences yields $\phi(\ker f)=p$ and $\phi(\coker f)=p$. 

Moreover, every subobject $L$ of $\ker f$ is a subobject of $M$, thus $\phi(L)\le \phi(M)=\phi(\ker f)$. 
Therefore $\ker f$ is $\phi$-semistable and belongs to $\mathcal{A}_p$. 
Dually we show that $\coker f$ also belongs to $\mathcal{A}_p$.
This finishes the proof.
\end{proof}

\begin{rmk}\label{simplestable}
It is easy to see that the $\phi$-stable objects with phase $ p$ are exactly the simple objects of the abelian category $\mathcal{A}_p$.
Moreover, the proof establishes again the parts (b) and (c)  of Theorem~\ref{homzero}.
\end{rmk}

%%%%%%%%%%%%%%%%%%%%%%%%%%%%%%%%%%%%%%
%%%%%%%%%%%%%%%%%%%%%%%%%%%%%%%%%%%%%%
\section{Maximal green sequences and stability functions}
%%%%%%%%%%%%%%%%%%%%%%%%%%%%%%%%%%%%%%
%%%%%%%%%%%%%%%%%%%%%%%%%%%%%%%%%%%%%%

In the previous section we discussed how a stability function $\phi: \A \to \P$ induces a torsion pair $(\T_p,\F_p)$ in $\A$ for each phase $p\in\P$. 
Moreover, as noted in \cite[Section 3]{BKT}, it is easy to see that if $p\leq q$ in $\mathcal{P}$, then $\T_p\supseteq\T_q$ and $\F_p\subseteq\F_q$.  
Since $\P$ is totally ordered, every stability function $\phi$ yields a (possibly infinite) chain of torsion classes in $\A$.
In this section we are mainly interested in the different torsion classes induced by $\phi$. 
We therefore define, for a fixed stability function $\phi: \A \to \P$, an equivalence relation on $\mathcal P$ by $p \sim q$ when $\mathcal{T}_p =\mathcal{T}_q$ and consider the equivalence classes $\P/\sim$.  

Of particular importance is the case where the chain of equivalence classes $\P/\sim$ is finite, not refinable, and represented by elements $p_0 > \ldots > p_m \in \P$ such that $\T_{p_0} = \{0\}$ and $\T_{p_m} = \A$:

\begin{defi}
A \textit{maximal green sequence} in $\A$ is a finite sequence of torsion classes $0=\X_0\subsetneq \X_1 \subsetneq \dots \subsetneq \X_{n-1}\subsetneq \X_n=\A$ such that for all $i\in\{1, 2, \dots, n\}$, the existence of a torsion class $\X$ satisfying $\mathcal{X}_i\subseteq\mathcal{X}\subseteq\mathcal{X}_{i+1}$ implies $\mathcal{X}=\mathcal{X}_i$ or $\mathcal{X}=\mathcal{X}_{i+1}$.
\end{defi}

\begin{rmk}
Note that this definition is not the original definition of maximal green sequence given by Keller in \cite{KellerMaximalGreenSequences}.
However the equivalence between both definitions follows directly from \cite[Proposition 4.9]{BSTw&c}.
\end{rmk}

Our aim is to establish conditions when the chain of torsion classes induced by a stability function is a maximal green sequence.
Observe first that if $\phi: \A \to \P$ is a stability function and the totally ordered set $\P$ has a maximal element $\overline{p}$, then $\T_{\overline{p}}$ is the minimal element in the chain of torsion classes induced by $\phi$.  

\begin{lem}\label{supremum}
Let $\phi: \A \to \P$ be a stability function.  
\begin{enumerate}[label=(\alph*)]
\item If $\P$ has a maximal element $\overline{p}$, then $\T_{\overline{p}}\neq \{0\}$ if and only if $\overline{p}\in\phi(\A)$.
\item If the set of equivalence classes $\P/\sim$ is finite and the maximal element of $\P$ does not belong to the image of $\phi$, then there exists some $p\in \P$ different from $\overline{p}$ such that $\T_p=\{0\}$.
\end{enumerate} 
\end{lem}
\begin{proof}
(a) 
Suppose that $\overline{p}$ is a maximal element in $\P$. 
If $\T_{\overline{p}}\neq \{0\}$, then there exists a non-zero object $M$ in $\T_{\overline{p}}$.
If $M'$ is the maximally destabilising quotient of $M$, we know that $\phi(M')\geq \overline{p}$. 
Since $\overline{p}$ is the maximal element of $\P$, we have   $\phi(M')=\overline{p}$ and thus $\overline{p}\in \phi(\A)$.

Conversely, if $\phi(M)=\overline{p}$, then it follows from the maximality of $\overline{p}$ that $\phi(L)\leq\phi(M)=\overline{p}$ for every non-trivial subobject $L$ of $M$. 
Thus $M$ is a $\phi$-semistable object, whence $M\in\A_{\overline{p}}\subset\T_{\overline{p}}$.
\medskip

(b) 
By assumption, the chain of torsion classes induced by $\phi$ is finite, say $$
\T_{p_0}\subsetneq \T_{p_1} \subsetneq \cdots \subsetneq \T_{p_n}.
$$

If $\T_{p_0}\neq\{0\}$, choose a non-zero object $M$ in $\T_{p_0}$.  
Let $M'$ be the maximally destabilising quotient of $M$, thus $M'\in\T_{p_0}$ and $\phi(M')\geq p_0$. 
Since the maximal object of $\P$ does not belong to the image of $\phi$, there exists a $p\in \P$ with $p>\phi(M')$.  
It follows that $M'\notin \T_{p}$, while $\T_p\subseteq \T_{p_0}$, contradicting the minimality of $\T_{p_0}$. 
Thus $\T_{p_0}=\{0\}$.  
\end{proof}

Following Engenhorst \cite{E14}, we call a stability function $\phi : \A \rightarrow \mathcal{P}$ \textit{discrete at $p$} if two $\phi$-stable objects $M_1, M_2$ satisfy $\phi(M_1)=\phi(M_2)=p$ precisely when $M_1$ and $M_2$ are isomorphic in $\A$. 
Moreover, we say that $\phi$ is \textit{discrete} if $\phi$ is discrete at $p$ for every $p\in\P$. 

\begin{prop}\label{mutation}
Let $\phi: \A \to \P$ be a stability function and $p,q\in\P$ such that $\T_p\subsetneq\T_q$.  
Then the  following statements are equivalent:
\begin{enumerate}[label=(\alph*)]
\item There is no $r\in\P$ such that $\T_p\subsetneq\T_r\subsetneq\T_q$, and $\phi$ is discrete at every $q'$ with $q'\sim q$.
\item There is no torsion class $\T$ such that $\T_p\subsetneq\T\subsetneq\T_q$. 
\end{enumerate}
\end{prop}
\begin{proof}
(a) implies (b): 
Suppose that $\T$ is a torsion class such that $\T_{p}\subsetneq\T\subseteq \T_{q}$. 
Then there exists an object $M\in\T\setminus\T_{p}$. 
Let $M'$ be the maximally destabilising quotient of $M$. 
Then $\phi(M')\geq q$ because $M\in \mathcal{T}\subsetneq\mathcal{T}_{q}$.  
Consequently, $\T_{\phi(M')}\subseteq\T_q$.  
On the other hand, $\phi(M')<p$ because $M\not\in \mathcal{T}_{p}$. 
Moreover, since $M'$ is $\phi$-semistable by Theorem~\ref{HN}, $M'\in\T_ {\phi(M')}\setminus\T_p$.  
Consequently, $\T_p\subsetneq \T_{\phi(M')}\subseteq \T_q$.
It thus follows from our assumption that $\T_{\phi(M')}=\T_q$. 

Now, Theorem~\ref{JH} implies the existence of a $\phi$-stable object $M''$ such that $\phi(M'')=\phi(M')$, which is unique since $\phi$ is discrete.  
Using Theorem~\ref{JH} again, $M'$ can be filtered by $M''$. 
In particular $M''$ is a quotient of $M$, and thus $M''\in\T$. 

Consider a $\phi$-stable object $X$ in $\A_{\geq \phi(M'')}$. 
In particular $X\in\T_{\phi(M'')}=\T_q$.  
If $\phi(X)=\phi(M'')$, then $X$ is isomorphic to $M''$ by the discreteness, and $X\in \T$.  
Else $\phi(X)>\phi(M'')$, and $M''\in\T_{\phi(M'')}\setminus\T_{\phi(X)}$. 
Therefore, $\T_{\phi(X)}\subsetneq \T_{\phi(M'')}=\T_q$, which implies by assumption, that $\T_{\phi(X)}\subseteq \T_p\subseteq \T$. 
In particular, $X\in\T$.  
Since $\T$ is a torsion class, this implies that $\A_{\geq \phi(M'')}\subseteq \T$, and furthermore $$\T_q=\T_{\phi(M'')}=\text{Filt}(\A_{\geq\phi(M'')})\subseteq \T.$$  
This shows $\T_q=\T$.

\bigskip
(b) implies (a): 
The fact that there is no $r\in\P$ such that $\T_p\subsetneq\T_r\subsetneq\T_q$ is immediate.  
To show that $\phi$ is discrete, assume that there exist two non-isomorphic $\phi$-stable objects $M$ and $N$ such that $\phi(M)=\phi(N)=q'$, with $q'\sim q$. 
Consider the set $\T=\text{Filt}(\A_{\geq q} \setminus \{M\})$.
We will show that $\T$ is a torsion class such that $\T_p\subsetneq \T\subsetneq \T_q$, a contradiction to our hypothesis.

First, because $\T_p\subsetneq\T_q=\T_{q'}$, we have $q<p$. Since $\T_p=\text{Filt}(\A_{\geq p})$, we have $N\notin \T_p$, so $\T_p\subsetneq \T$.  
Furthermore, $M\notin \T$ by hypothesis.  
Since $M\in\T_q$, this shows $\T\subsetneq\T_q$.  
Thus $\T_p\subsetneq\T\subsetneq\T_q$.    

We now show that $\T=\text{Filt}(\A_{\geq p}\cup \{N\})$ is a torsion class, that is, $\T$ is closed under extensions and quotients. By definition, $\T$ is closed under extensions.  
To show that $\T$ is closed under quotients, suppose that $$T\rightarrow T'\rightarrow 0$$ is an exact sequence in $\A$ and $T\in\T$.  

If $T\in\T_p$, then $T'\in\T_p$ since $\T_p$ is a torsion class and therefore $T'\in\T$.

Else, $T\in\T\setminus\T_p$.  
Let $Q$ be the maximally destabilising quotient of $T$. Since $\T\notin\T_p$, we have $\phi(Q)<p$.  
Moreover, $\phi(Q)\geq q$ since $T\in\T\subsetneq\T_q$.  Consequently, $q\leq\phi(Q)<p$, and it follows from our hypothesis that $\phi(Q)=q$ (otherwise $\T_p\subsetneq\T_{\phi(Q)}\subsetneq \T_q$).  
So $Q\in\T_q=\T_{q'}$. This shows in particular that $q=q'$.
Indeed, if $q<q'$, then the fact that $Q$ is $\phi$-semistable leads to $Q\in\T_q\notin\T_{q'}$, a contradiction.  
Similarly, if $q'<q$, then $N\in\T_{q'}\notin\T_q$, again a contradiction.  So $q=q'$, and consequently $Q,N\in\A_q$.

Now, suppose that
$$
0=T_0\subsetneq T_1\subsetneq T_2\subsetneq \cdots \subsetneq T_{n-1}\subsetneq T_n=T
$$
is a Harder-Narasimhan filtration of $T$, as in Theorem~\ref{HN}.  
In particular, $Q\cong T/T_{n-1}$ and $$q=\phi(Q)<\phi(T_{n-1}/T_{n-2}) <\cdots< \phi(T_{2}/T_{1})<\phi(T_1/T_0).$$  
Consequently, $T_i/T_{i-1}\in\A_p$ for all $i\leq n-1$, while $\phi(Q)=q$.  
Now, $Q$ is a $\phi$-semistable object since is the maximally destabilising quotient of $T$.
Therefore Theorem \ref{JH} implies that $Q \in \Filt(\{M,N\})$. 
But, at the same time we have by hypothesis that $M$ is not a composition factor of $T$. 
In particular $M$ is not a composition factor of $Q$, which implies that $Q \in \Filt(\{N\}) \subset \T$.

Now, let $Q'$ be the maximally destabilising quotient of $T'$. 
Since $Q$ is the maximally destabilising quotient of $T$, we have $\phi(Q')\geq \phi(Q)$.  
If $\phi(Q')>\phi(Q)$, then $\phi(Q')\geq p$, and $T'\in\T_p\subsetneq \T$.  
Else, $\phi(Q')=\phi(Q)$, and it follows from the fact that $Q$ is the maximally destabilising quotient of $T$ that the epimorphism from $T$ to $Q'$ factors through $Q$, and thus there exists an epimorphism $f:Q\rightarrow Q'$ in $\A$, and thus in $\A_q$.

Recall from  Proposition~\ref{fix-slope} that $\A_q$ is an abelian category whose $\phi$-stable objects coincide with the simple objects by Remark~\ref{simplestable}.  
Consequently, it follows from the existence of the epimorphism $f:Q\rightarrow Q'$ and the fact that $Q$ is filtered by the $\phi$-stable object $N$ that $Q'\in\text{Filt}(\{N\})$.  

Let
$$
0=T'_0\subsetneq T'_1\subsetneq T'_2\subsetneq \cdots \subsetneq T'_{m-1}\subsetneq T'_m=T'
$$
be the Harder-Narasimhan filtration of $T'$. 
Then $Q'\cong T'/T'_{m-1}$ and $$q=\phi(Q')<\phi(T'_{m-1}/T'_{m-2}) <\cdots< \phi(T'_{2}/T'_{1})<\phi(T'_1/T'_0).$$
Consequently, $T'_i/T'_{i-1}\in\A_p$. 
Since $Q'$ is filtered by $N$, this implies that $T'\in\text{Filt}(\A_p\cup\{N\})=\T$. 
This finishes the proof.
\end{proof}

%{\color{orange} We may skip altogether the next remark... I don't know how to say that our result appeared before in the arXiv also.}

We are now able to characterize the stability functions inducing maximal green sequences in $\A$.

\begin{theorem}\label{maximalgreensequences}
Let $\phi:\A\ra \P$ be a stability function. 
Suppose that $\P$ has no maximal element, or that the maximal element of $\P$ is not in $\phi(\A)$. 
Then $\phi$ induces a maximal green sequence if and only if $\phi$ is a discrete stability function inducing finitely many equivalent classes on $\P/\sim$.
\end{theorem}
\begin{proof}
Suppose that $\phi$ induces a maximal green sequence, say 
$$
\{0\}=\T_{p_0}\subsetneq \T_{p_1} \subsetneq \cdots \subsetneq \T_{p_n}=\A.
$$
In particular, there are only finitely many equivalence classes in $\P/\sim$.
Moreover, it follows from Proposition~\ref{mutation} that $\phi$ is discrete.

Conversely, suppose that $\phi$ is a discrete stability function inducing finitely many equivalent classes on $\P/\sim$.  
So we get a (complete) chain of torsion classes
$$
\T_{p_0}\subsetneq \T_{p_1} \subsetneq \cdots \subsetneq \T_{p_n}.
$$
induced by $\phi$.
The discreteness of $\phi$ implies by Proposition~\ref{mutation} that this chain of torsion classes is maximal. 
Moreover, it follows from Lemma~\ref{supremum} that $\T_{p_0}=\{0\}$. It remains to show that $\T_{p_n}=\A$. 
If $M\in\A$ but $M\notin\T_{p_n}$, then the maximally destabilising quotient $M'$ of $M$ satisfies $\phi(M')<p_n$.
Since $M'\in \T_{\phi(M')}$, it follows that $\T_{p_n}\subsetneq \T_{\phi(M')}$, a contradiction to the maximality of $\T_{p_n}$. 
So $\T_{p_n}=\A$. 
\end{proof}

As an immediate corollary we have the following result, which is of importance for the study of the representation theory of the so-called $\tau$-tilting finite algebras.

\begin{cor}
Let $\A$ be an abelian category having only finitely many torsion classes. 
Then every discrete stability function $\phi: \Obj^*(\A) \to \P$ induces a maximal green sequence. 
\end{cor}

\begin{ex}
We illustrate by the following example that non-linear stability functions allow sometimes to describe all torsion classes, which would not have been possible using linear stability conditions. Consider the Kronecker quiver

$$\xymatrix{
 Q: \quad 1 \ar@<-.5ex>[r] \ar@<.5ex>[r] & 2
}$$
It is well-known that the indecomposable representations of $Q$ are parametrized by two discrete families $P_n$ and $I_n$, for $n \in \N$, of dimension vectors $(n, n+1)$ and $(n+1,n)$, respectively, together with a $\mP_1(k)-$family of representations $R_{\lambda,n}$ of dimension vector $(n,n)$, with $\lambda \in \mP_1(k), n\in \N$, for an algebraically closed field $k$.

We order the indecomposables by their slope

$$\phi(V) = \mbox{$\frac{n_1}{n_2}$}\; \mbox{ if dim }V = (n_1,n_2)$$
and thus obtain a stability function
$$\phi: {\rm rep }\; Q \to \R \cup \{ \infty \}.$$
It is known that one obtains all functorially finite torsion classes of rep $Q$ in the form $\T_p$ for some $p \in \R\cup \{ \infty \}$.
For more details on this, see \cite{AIR, DIJ}.

However, there are lots of torsion classes for rep $Q$ that are not functorially finite, they are given by selection of indecomposables as follows:
Let $S$ be any subset of $\mP_1(k)$, then the additive hull of all indecomposables $R_{\lambda,n}$ and $I_n$, for $n \in \N$ and $\lambda \in S$, forms a torsion class which we denote by $\T_S$. Every not functorially finite torsion class of rep $Q$ is of this form for some set $S$, and we can certainly not obtain these classes by a linear stability function since the elements $R_{\lambda,n}$ where $\lambda$ is in $S$ share the same dimension vector with those where $\lambda$ does not lie in $S$. 

We therefore define a set $\P = \R \cup \{ \infty \} \cup \{1^*\}$ where we add a new element, $1^*$, as a double of $1$, at the same order relative to the other elements $x \neq 1$, but we agree on setting $ 1^* < 1$. 
Thus $\P$ is totally ordered, and we define a stability function 
$$\phi^*: {\rm rep }\; Q \to \P$$
by the following values on the indecomposables:
\begin{equation*}
  \phi^*(V) =\left\{
  \begin{array}{@{}ll@{}}
    \mbox{$\frac{n_1}{n_2}$}\;  & \text{ if dim } V = (n_1,n_2) \text{ and } n_1 \neq n_2\\
    1 & \text{ if } V = R_{\lambda,n} \text{ and } \lambda \in S 
   \\
  1^* & \text{ if } V = R_{\lambda,n} \text{ and } \lambda \not\in S 
  \end{array}\right.
\end{equation*} 
Using this setting, one obtains the torsion class $\T_S$ as $\T_1$ with respect to the element $p=1 \in \P$.
\end{ex}

%%%%%%%%%%%%%%%%%%%%%%%%%%%%%%%%%%%%%%%%%%%%%%%%%%%%%%%%%%%%%%%%%%%%
\section{Paths in the wall and chamber structure}\label{Sc:MGSmodcat}
%%%%%%%%%%%%%%%%%%%%%%%%%%%%%%%%%%%%%%%%%%%%%%%%%%%%%%%%%%%%%%%%%%%

In this section we focus on  abelian length categories $\A$ with finitely many simple objects, that is $rk(K_0(\A))=n$ for some $n \in \N$. We provide a construction of stability functions on $\A$  that conjecturally induce all its maximal green sequences.
These stability functions are induced by certain curves, called \textit{red paths}, in the wall and chamber structure of $\A$, described in \cite{BSTw&c} when $\A$ is the module category of an algebra.
In particular we show that red paths give a non-trivial compatibility between the stability conditions introduced by King in \cite{Ki} and the stability functions introduced by Rudakov in \cite{Ru}.
As a consequence, we show that the wall and chamber structure of an algebra can be recovered using red paths.

%In this section we study paths in the wall and chamber structure of an algebra using the description given in the previous section to show that every maximal green sequence can be induced by a path on it. As an application we show that certain algebras admit no maximal green sequences on its module category. 

%{\color{orange} We should recall the construction of the wall and chamber structure.}

%%%%%%%%%%%%%%%%%%%%%%%%
\subsection{The wall and chamber structure of an abelian category}\label{wallandchamber}

One of the main motivations of Rudakov to introduce stability functions was to generalise the stability conditions introduced by King in \cite{Ki}.
The definition of stability conditions given by King is the following.

\begin{defi}\cite[Definition 1.1]{Ki}\label{stability}
Let $\theta$ be a vector of $\mathbb{R}^n$ and $M$ be an object in $\A$.
Then $M$ is called \textit{$\theta$-stable} (or \textit{$\theta$-semistable}) if $\langle \theta, [M]\rangle=0$ and $\langle \theta, [L]\rangle<0$ ($\langle \theta, [L]\rangle\leq 0$, respectively) for every proper subobject $L$ of $M$. 
\end{defi}

One can see Definition \ref{stability} as a forking path:
either one fixes a vector $\theta$ and studies the category of $\theta$-semistable objects, or one fixes an object $M$ and studies the vectors $\theta$ turning $M$ $\theta$-semistable. 
The wall and chamber structure of $\A$ is defined taking the second option.

\begin{defi}
The \textit{stability space} of an object $M$ of $\A$ is $$\mathfrak{D}(M)=\{\theta\in\mathbb{R}^n : M \text{ is $\theta$-semistable}\}.$$
Moreover the stability space $\mathfrak{D}(M)$ of $M$ is said to be a \textit{wall} when $\mathfrak{D}(M)$ has codimension one. 
In this case we say that $\mathfrak{D}(M)$ is the wall defined by $M$.
\end{defi}

Note that not every $\theta \in \mathbb{R}^n$ belongs to the stability space $\mathfrak{D}(M)$ for some non-zero object $M$. For instance, is easy to see that $\theta=(1,1, \dots, 1)$
is an example of such a vector for every  $\A$. 
This leads to the following definition.

\begin{defi}
Let $\A$ be an abelian length category such that $rk(K_0(\A))=n$ and
$$\mathfrak{R}=\mathbb{R}^n\setminus\overline{\bigcup\limits_{\substack{ 0 \neq M\in \A}}\mathfrak{D}(M)}$$
be the maximal open set of all $\theta$ having no $\theta$-semistable objects other than the zero object. 
Then an $n-$dimensionional  connected component $\mathfrak{C}$ of $\mathfrak{R}$ is called a \textit{chamber} and this partition of $\mathbb{R}^{n}$ is known as the \textit{wall and chamber structure} of $\A$.
\end{defi}

\subsection{Red paths}

Let $\A$ be an abelian length category of rank $n$ as before, and let $\gamma:[0,1]\rightarrow \mathbb{R}^n$ be a continuous function such that $\gamma(0)=(1,\dots, 1)$ and $\gamma(1)=(-1,\dots,-1)$.
If we fix an object $M$ in $\A$, $\gamma(t)$ induces a continuous function $\rho_M:[0,1]\rightarrow \mathbb{R}$ defined as $\rho_M(t)=\langle \gamma(t) , [M] \rangle$. 
Note that $\rho_M(0)>0$ and $\rho_M(1)<0$. 
Therefore, for every object there is at least one $t\in(0,1)$ such that $\rho_M(t)=0$. 
This leads to the following definition of red paths:

\begin{defi}\label{greenpaths}
%Fix $\A$ to be an abelian length category. 
A continuous function $\gamma: [0,1]\ra\mathbb{R}^n$ in the wall and chamber structure of $\A$ is a \textit{red path} if the following conditions hold:
\begin{itemize}
\item $\gamma(0)=(1,\dots, 1)$;
\item $\gamma(1)=(-1,\dots,-1)$;
\item for every non-zero object $M$ there is a unique $t_M \in [0,1]$ such that $\rho_M(t_M)=0$.
\end{itemize}
\end{defi}

\begin{rmk}
In \cite[Section 4]{BSTw&c} the notion of $\D$-generic paths in the wall and chamber of an algebra $A$ is studied. 
In particular, every wall crossing of a $\D$-generic path is either green or red. 
We use the name \textit{red} paths here because every wall crossing is red in the sense of \cite{BSTw&c}.
\end{rmk}

\begin{rmk}
Note that, by definition, red paths can pass through the intersection of walls, which is not allowed in the definition of Bridgeland's \textit{$\D$-generic paths} (see \cite[Definition 2.7]{B16}) nor Engenhorst's  \textit{discrete paths} (see \cite{E14}).
\end{rmk}

\bigskip
Another key difference between the red paths and the other paths cited above is that red paths take account of crossing of all hyperplanes, not only  the walls. In the next proposition we show that we can recover the information of  crossings from the stability structure induced by the path. 

\begin{lem}\label{darkside}
Let $\gamma$ be a red path. Then $\langle \gamma(t), [M] \rangle<0$ if and only if $t>t_M$.
\end{lem}

\begin{proof}
This is a direct consequence of the definition of red path and the fact that the function $\rho_M$ induced by $\gamma$ and $M$ is continuous. 
\end{proof}

The following result shows that each red path $\gamma$ yields a stability function $\phi_\gamma$ keeping track of the walls that are crossed by $\gamma$.

\begin{theorem}\label{paths}
Let $\A$ be an abelian length category such that $rk(K_0(A))=n$. 
Then every red path $\gamma:[0,1]\to \mathbb{R}^n$ induces a stability function $\phi_{\gamma}:\A\to [0,1]$ defined by $\phi_{\gamma}(M)=t_M$, where $t_M$ is the unique element in $[0,1]$ such that $\langle \gamma(t_M), [M] \rangle=0$. 
Moreover $M$ is $\phi_{\gamma}$-semistable if and only if $M$ is $\gamma(t_M)$-semistable. 
\end{theorem}
 
\begin{proof}
Let $\gamma$ be a red path in $\mathbb{R}^n$. 
First, note that $\phi_{\gamma}$ is a well defined function by the definition of red path. 
We want to show that $\phi_\gamma$ induces a stability structure in $\A$. 
It follows from lemma \ref{darkside} that $\langle \gamma(t), [M] \rangle<0$ if and only if $t>t_M$ and $\langle \gamma(t), [M] \rangle>0$ if and only if $t<t_M$. 

Consider a short exact sequence 
$$0\ra L\ra M\ra N\ra 0$$ 
and suppose that $\phi_\gamma([L])<\phi_\gamma([M])$. 
Then $\langle \gamma(t_M), [M] \rangle=0$ and $\langle \gamma(t_M), [L] \rangle<0$. 
Therefore 
$$\langle \gamma(t_M), [N] \rangle=\langle \gamma(t_M), [M]-[L] \rangle=\langle \gamma(t_M), [M] \rangle-\langle \gamma(t_M), [L] \rangle>0.$$ 
Hence $\phi_{\gamma}([L])<\phi_\gamma([M])<\phi_\gamma([N])$.
The other two conditions of the see-saw property are proved in a similar way, which shows that $\phi_\gamma$ is a stability function by Definition \ref{seesaw}.

Now we prove the moreover part of the statement.
Let $M$ be a non-zero object of $\A$ and suppose that $M$ is $\phi_{\gamma}$-semistable. 
Then $\phi_{\gamma}(L)\leq\phi_{\gamma}(M)$ (i.e., $t_L\leq t_M$) for every proper subobject $L$ of $M$.
Therefore $\langle \gamma(t_M), [L] \rangle\leq 0$ by lemma \ref{darkside}.
Thus $M$ is $\gamma(t_M)$-semistable.

On the other hand, suppose that $M$ is $\gamma(t_M)$-semistable and $L$ is a proper object of $M$. 
Then $\langle \gamma(t_M), [L] \rangle\leq 0$, hence $t_L\leq t_M$ by Lemma \ref{darkside}. 
Therefore $M$ is $\phi_{\gamma}$-semistable. 
\end{proof}

As a consequence of the previous theorem we get the following result, in which we use the notations of Subsection~\ref{Sect:torsion} with $\P=[0,1]$.

\begin{prop}\label{green-to-red}
Let $\gamma$ be a red path and let $\phi_{\gamma}$ be the stability structure induced by $\gamma$. 
Then $\T_0=\A$ and $\T_1=\{0\}$.
\end{prop}

\begin{proof}
Let $\gamma$ be a red path and let $M$ a non-zero object in $\A$. 
Then $\langle \gamma(1), [M] \rangle=\langle (-1, -1,\dots, -1), [M]\rangle<0$. 
Hence $1\not\in \phi_{\gamma}(\A)$. Therefore Lemma \ref{supremum} implies that $\T_1=\{0\}$. 

Using dual arguments one can prove that $\F_0=\{0\}$.
Then $$\T_0 = \{X\in \A : \Hom_{\A}(X, 0)=0\}= \A.$$
This finishes the proof.
\end{proof}

\begin{cor}\label{maximal-paths}
Let $\gamma$ be a red path and let $\phi_{\gamma}$ be the stability function induced by $\gamma$. 
Then $\gamma$ induces a maximal green sequence if and only if the set $\mathcal{S}_\gamma$ of $\phi_\gamma$-stables objects is finite and they are such that $t_{M}\neq t_{N}$ for every pair of non-isomorphic $M,N\in \mathcal{S}_\gamma$. 
\end{cor}

\begin{proof}
The fact that if $\gamma$ induces a maximal green sequence, then the set $\mathcal{S}_\gamma$ of $\phi_\gamma$-stable objects has the properties indicated in the statement follow directly from Theorem \ref{maximalgreensequences}.

Now we show the other implication.
Since $\mathcal{S}_\gamma$ is finite, we can write it as $\mathcal{S}_\gamma = \{M_1, \dots, M_n\}$.
Without loss of generality we can suppose that $t_{M_i}\leq t_{M_j}$ if $i<j$. 
It is easy to see that the finiteness of $\mathcal{S}_\gamma$ implies that the chain of torsion classes induced by $\gamma$ is finite.
Moreover Proposition \ref{green-to-red} implies that $\T_0=\A$ and $\T_1=\{0\}$. 
Finally, we have that $t_{M_i}<t_{M_j}$ whenever $i<j$, then we have that $\phi_{\gamma}$ is a discrete. 
Therefore Theorem \ref{maximalgreensequences} implies that $\gamma$ induces a maximal green sequence. 
\end{proof}

\bigskip

Recall that Bridgeland associated in \cite[Lemma 6.6]{B16} a torsion class $\mathcal{T}_{\theta}$ to every $\theta \in \mathbb{R}^n$ as follows.
$$\mathcal{T}_{\theta}=\{M\in \A:\langle \theta,N\rangle \geq 0 \text{ for every quotient $N$ of $M$}\}$$
On the other hand, in subsection \ref{Sect:torsion} we have studied the torsion classes associated to stability functions.
Therefore, given a red path $\gamma$ it is natural to compare the torsion classes given by $\T_{\gamma(t)}$ and $\T_{t}$ for all $t\in [0,1]$.
This is done in the following proposition.

\begin{prop}\label{redtorsion}
Let $\gamma$ be a red path, $\T_t$ the torsion class associated to the stability function $\phi_\gamma$ and $\T_{\gamma(t)}$ as defined by Bridgeland.
Then $\T_t=\T_{\gamma(t)}$ for every $t\in[0,1]$.
\end{prop}

\begin{proof}
By definition we have that 
$$\mathcal{T}_{\gamma_t}=\{M\in \A:\langle \gamma_t,N\rangle \geq 0 \text{ for every quotient $N$ of $M$}\}.$$
Now, applying Lemma \ref{darkside} this can be rewriten as 
$$\T_{\gamma_t}=\{M\in\A \ : \phi_{\gamma_t}(N)\geq p \text{ for every quotient $N$ of $M$}\},$$
which is exactly $\T_t$ by Proposition \ref{eqdeftorsion}.
This finishes the proof. 
\end{proof}
\bigskip

In \cite{B16}, Bridgeland defined the $\D$-generic paths in order to show the consistency of a scattering diagram that he introduced for the module category of certain algebras. 
These scattering diagrams consists on the wall and chamber structure of the module category of the algebra, where each wall is decorated by an element of the motivic Hall algebra associated to the original algebra.
Now, in the following definition we recall the combinatorial properties of $\D$-generic paths.
By abuse of notation we call them $\D$-generic paths as well, even if we do not consider some of the geometrical and algebraic aspects of the original definition.

\begin{defi}\cite[\S 2.7]{B16}\label{defDgeneric}
We say that a smooth path $\gamma:[0,1]\to \mathbb{R}^n$ is a $\mathfrak{D}$-generic path if:
\begin{enumerate}
\item $\gamma(0)$ and $\gamma(1)$ do not belong to the stability space $\D(M)$ of a non-zero object $M$, that is, they are located inside some chambers;
\item $\gamma$ do not meet any cone of codimension greater that 1;
\item all intersections of $\gamma$ with a wall are transversal.
\end{enumerate}
\end{defi}

Note that the second condition of the previous definition implies that $\gamma$ do not pass through the intersection of two non-parallel walls. 
Given that every wall is induced by some object $M$, this implies that $\gamma$ crosses the intersection of the stability space of two non-isomorphic modules $M, N$ only if their elements in the Grothendieck group of the category $[M]$ and $[N]$ are parallel.

Also, if $\gamma$ crosses at $t_0$ a wall $\D(M)$ transversely it means that the vector $\gamma'(t_0)$ does not belong to the hyperplane in which the wall $\D(M)$ is contained. 
In other words, this means that $\gamma'(t_0)$ is not perpendicular to $[M]$. 

These remarks allow us to define $\D$-generic paths equivalentely as follows.

\begin{defi}\cite[\S 2.7]{B16}\label{defDgeneric1}
We say that a smooth path $\gamma:[0,1]\to \mathbb{R}^n$ is a $\mathfrak{D}$-generic path if:
\begin{enumerate}
\item $\gamma(0)$ and $\gamma(1)$ do not belong to the stability space $\D(M)$ of a non-zero object $M$, that is, they are located inside some chambers;
\item If $\gamma(t)$ belongs to the intersection $\D(M)\cap\D(N)$ of two walls, then the dimension vector $[M]$ of $M$ is a scalar multiple of the dimension vector $[N]$ of $N$;
\item whenever $\gamma(t)$ is in $\D(M)$, then $\langle \gamma'(t), [M] \rangle \neq 0$.
\end{enumerate}
\end{defi}

It is clear that every red path $\gamma'$ inducing a maximal green sequence in $\A$ satisfies condition 1.
Moreover, Theorem \ref{paths} and Corollary \ref{maximal-paths} say that $\gamma'$ has at most one $\phi_{\gamma'}$-stable object for every $t \in [0,1]$.
This implies in particular that $\gamma'$ is in the intersection of $\D(M_1)$ and $\D(M_2)$ if and only if $M_1$ and $M_2$ are filtered by the same $\phi_{\gamma'}$-stable module $M$. 
In particular this implies condition 2 of Definition \ref{defDgeneric1}.

On the other hand, one of the main results in \cite{BSTw&c} says that every maximal green sequence  is induced by a $\D$-generic path.
This lead us to the following conjecture.

\begin{conj}
Let $\A$ be an abelian length category of finite rank. 
Then every maximal green sequence in $\A$ is induced by a red path in the wall and chamber structure of $\A$.
\end{conj}

\bibliography{BibliografiaTreffinger}{}
\bibliographystyle{abbrv}

\end{document}